 \documentclass{amsproc}
\usepackage{xcolor}
\usepackage{amssymb}
\newtheorem{theorem}{Theorem}[section]
\newtheorem{lemma}[theorem]{Lemma}
\newtheorem{proposition}[theorem]{Proposition}

\usepackage{geometry}
 \usepackage{enumitem}
\theoremstyle{definition}
\newtheorem{definition}[theorem]{Definition}
\newtheorem{example}[theorem]{Example}

\theoremstyle{remark}
\newtheorem{remark}[theorem]{Remark}
\newtheorem{note}[theorem]{Note}

\numberwithin{equation}{section}
\usepackage{epsfig} 
\usepackage{epstopdf} 

\begin{document}

\title[Some remarks on the exponential separation and dimensions]{Some remarks on the exponential separation and dimension preserving approximation for sets and measures}



\author{Saurabh Verma}
\address{Department of Applied Sciences, IIIT Allahabad, Prayagraj, India 211015}
\email{saurabhverma@iiita.ac.in}
\author{Ekta Agrawal}
\address{Department of Applied Sciences, IIIT Allahabad, Prayagraj, India 211015}
%
\email{ekta.agrawal5346@gmail.com}
\author{Megala M}
\address{School of Sciences, NIT Andhra Pradesh, Tadepalligudem, India 534101}
%
\email{megala8niya@gmail.com}



\begin{abstract}
In the dimension theory of sets and measures, a recent breakthrough happened due to Hochman, who introduced the exponential separation condition (ESC) and proved the Hausdorff dimension result for invariant sets and measures generated by similarities on the real line. Following this groundbreaking work, we make a modest contribution by weakening the condition. Further, we define the modified ESC using the convex hull of the attractor and show that for homogeneous self-similar IFS on $\mathbb{R},$ both definitions coincide. We also define some sets in the class of all nonempty compact sets using the Assouad and Hausdorff dimensions and subsets of measures in the space of Borel probability measures on $\mathbb{R}^m$ using the $L^q$ dimension and the Rajchman property, and prove their density in the respective spaces.
\end{abstract}

\maketitle

\section{Introduction}
 The dimension estimation of a fractal set and associated measure constructed using an iterated function system (IFS) (see \cite{H}) has attracted a lot of researchers over the last century. Depending on the properties of the contractive transformations, such as similarity (also called similitude), affine, or conformal in $\mathbb{R}^m,$ the fractal set and the measure are characterized as self-similar, self-affine, and self-conformal, respectively. Several dimensions are introduced in the literature, as an example Hausdorff dimension $(\dim_{\mathcal{H}}) $ \cite{Fal, Mat3}, Assouad dimensions $(\dim_{\mathcal{A}})$ \cite{Fraser}, $L^q$ dimension $(\dim_{L^q})$ \cite{O1, Sh1} for the invariant set and measure. Initially, we mention recent progress being made in dimension estimation. Let $A$ and $\mu$ be the fractal set and measure associated to some IFS on $\mathbb{R}^m$. The trivial bound is obtained as 
 \begin{align*}
     &\dim_{\mathcal{H}}A\le\min\{m, \dim_SA\},\\
     &\dim_{\mathcal{H}}\mu\le\min\{m, \dim_S\mu\},
 \end{align*}
 where $\dim_SA$ and $\dim_S\mu$ are the similarity dimensions of $A$ and $\mu,$ respectively. In general, calculating dimensions is a tedious task. In view of this, several separation conditions are defined in an IFS, such as strong separation condition (SSC), open set condition (OSC), strong open set condition (SOSC), and exponential separation condition (ESC), for details \cite{Fraser,MHOCHMAN,Schief}. In general, SSC$\implies$SOSC$\implies$OSC$\implies$ESC, concluding that ESC is the weakest. But the converse need not hold in general. The work conducted so far in the context of dimensional analysis of sets and measures under the separation condition is listed as follows:
 \begin{itemize}
     \item Under OSC, $\dim_{\mathcal{H}}A=\min\{m, \dim_SA\},$ see \cite{H,M}.
     \item IFS with bi-Lipschitz mappings $(s_i\|x-y\|_2\le\|f_i(x)-f_i(y)\|_2\le t_i\|x-y\|_2)$ on $\mathbb{R}^m$ and SSC, we have $s\le\dim_{\mathcal{H}}A\le t,$ where $s$ and $t$ satisfies $\sum_{i=1}^Ns_i^s=1$ and $\sum_{i=1}^Nt_i^t=1,$ respectively, see \cite[Proposition 9.6-9.7]{Fal}. It is noted that the upper bound holds without any separation condition, but the lower bound does not hold even if we assume OSC. Additionally, it does not serve as a bound for the Assouad dimension. For example, visit \cite[Section 8.5]{Fraser}. 
     \item Total overlap implies $\dim_{\mathcal{H}}A<\min\{m, \dim_SA\}.$
     \item Generally, $\dim_{\mathcal{H}}A
     \le\dim_{\mathcal{A}}A,$ but if $A$ is Ahlfors regular, then $\dim_{\mathcal{H}}A=\dim_{\mathcal{A}}A,$ see \cite[Theorem 6.4.1]{Fraser}. IFS with similitudes on $\mathbb{R},$ either $\dim_{\mathcal{H}}A=\dim_{\mathcal{A}}A$ or $\dim_{\mathcal{A}}A=1.$
     \item In $\mathbb{R}$ with similitudes, we have ESC$\implies\dim_{\mathcal{H}}A=\min\{1, \dim_SA\},$ or we can say that dimension drop $\implies$ super-exponential condensation, see \cite{MHOCHMAN} (Similar result in $\mathbb{R}^m$ is shown in \cite{MHOCHMAN1}). Rapaport \cite{AR1} extends the result of Hochman \cite{MHOCHMAN} to algebraic contraction ratios and arbitrary translations. We refer the reader to a few more articles in this direction \cite{AR2, PPV1}.
     \item In $\mathbb{R}$ with similarities and ESC, Shmerkin computed $\dim_{L^q}\mu=\min\{1, \tilde\tau(q)/(q-1)\},$ see \cite[Theorem 6.6]{Sh1}.
     \item IFS with similarity and OSC implies $\dim_{\mathcal{H}}\mu=\min\{m, \dim_S\mu\}$, see \cite{DGSH}. In this case, under SSC, Fraser and Howroyd computed $\dim_{\mathcal{A}}\mu,$ see \cite[Theorem 7.4.1]{Fraser}. 
 \end{itemize}
 It is well known that the Hausdorff dimension decreases under Lipschitz mapping (\cite[Corollary 2.4]{Fal}), but this is not the case with the Assouad dimension (\cite[Corollary 7.2.2]{Fraser}). Bedford \cite{Bed} in his PhD thesis, studied the box dimension of self-similar sets. Recent work in the dimension estimation of self-similar sets can be explored in \cite{VP1}. Initially, a special class of self-affine sets known as self-affine carpets, such as Bedford-McMullen carpets, is studied; for details, see \cite{Fraser1}. Several recent works in the direction of estimation of the Hausdorff dimension of self-affine sets and measures can be found in \cite{BB,Fal1,Hochman2,IDM,BS1}. \\
 In geometric measure theory, the Fourier transform has emerged as a powerful technique for determining the Hausdorff dimension. The interplay between dimension theory and Fourier analysis was first observed by Kaufman \cite{Kaufman}, who proved one part of Marstrand's projection theorem using the Fourier transform method. Furthermore, this technique allows researchers to introduce the Fourier dimension $(\dim_{\mathcal{F}})$ of any subset of $\mathbb{R}^m.$ For any Borel subset $A$ of $\mathbb{R}^m,$ we have $\dim_{\mathcal{F}}(A)\le \dim_{\mathcal{H}}(A)$, and $\dim_{\mathcal{F}}(A)= \dim_{\mathcal{H}}(A)$ iff $A$ is Salem set. Interested readers may refer to the book \cite{Mat3}, and the references therein for details. Recent work on the Fourier decay of self-similar measures and homogeneous self-affine measures can be referred to \cite{BS,BS12}.
   
 The primary theorems in this article include the following.
 \begin{itemize}
     \item The weaker version of the Hochman result \cite[Theorem 1.4]{MHOCHMAN1} is shown in Theorem \ref{th310}.
     \item The modified ESC is being defined using the Hausdorff metric and compared with the existing ESC, along with some interesting remarks. 
     \item Several dense subsets are explored that preserve the Assouad dimension and $L^q$ dimension.
 \end{itemize}
Finally, the identified research gap and potential future comments are outlined in the last Section \ref{Sec4}. 

\section{Preliminaries} \label{Sec2}
Let $\chi(\mathbb{R}^m)$ denote the collection of all nonempty compact subsets of $\mathbb{R}^m$ in the Euclidean norm $\|.\|_2.$ The Hausdorff distance $\mathfrak{H}$ between any $A,B\in\chi(\mathbb{R}^m)$ is elaborated as 
$$\mathfrak{H}(A,B):=\inf\{\epsilon>0: A\subset B_\epsilon \text{ and } B\subset A_\epsilon\},$$ where $A_\epsilon$ denotes the $\epsilon$-neighbourhood of $A.$ It is also redefined as 
$$\mathfrak{H}(A,B):=\max\Big\{\sup_{a\in A}\inf_{b\in B}\|a-b\|_2, \sup_{b\in B}\inf_{a\in A}\|a-b\|_2\Big\}.$$
\begin{definition}\cite{Mat3}
The Fourier transform of any Borel probability measure $\eta$ in $\mathbb{R}^m$ is defined as 
\[
\hat{\eta}(\zeta)= \int e^{-2\pi i \zeta .y} d\eta(y),~~~~ \zeta \in \mathbb{R}^m.
\]
\end{definition}
Recall that the Fourier transform of convolution $\eta*\nu$ has the property  $\widehat{\eta*\nu}(\zeta)=\hat{\eta}(\zeta)\hat{\nu}(\zeta),$ for details; see \cite{Mat3}. It is said that $\eta$ is a Rajchman measure if $\lim_{|\zeta| \to \infty}|\hat{\eta}(\zeta)|= 0.$ 
 
\begin{definition}\cite{Fal}
Let $E$ be a subset of $\mathbb{R}^m$ and $r$ be any non-negative number. Then, the $r$-dimensional Hausdorff measure of $E$ is defined as $$\mathcal{H}^r(E)=\lim_{\delta\to 0^+}\mathcal{H}_{\delta}^r(E),$$
where $\mathcal{H}_{\delta}^r(E)=\inf\{\sum_{i=1}^{\infty}(diam(U_i))^r:~ \{U_i\} \text{ is a } \delta\text{-cover of } E \}$ and $diam(U_i)=\sup_{x,y\in U_i}\|x-y\|_2.$ The Hausdorff dimension of $E$ is defined as $$\dim_{\mathcal{H}}E=\inf\{r\geq0:~\mathcal{H}^r(E)=0\}.$$
\end{definition}
The approximation of a compact set that preserves dimension is explored in \cite{VM1}.
\begin{definition}\cite{Fraser}
Let $E$ be a non-empty subset of $\mathbb{R}^m.$ The Assouad dimension of $E$ is defined as
\begin{align*}
   \dim_{\mathcal{A}}E=\inf\Bigg\{\beta\ge0: \exists \text{ a constant }&C>0\text{ such that }\forall~ 0<r<R\text{ and }x\in E,\\&N_r(B(x,R)\cap E)\le C\Big(\frac{R}{r}\Big)^\beta\Bigg\}, 
\end{align*}
where $B(x,R)=\{y\in\mathbb{R}^m:\|y-x\|_2\le R\}$ and $N_r(B(x,R)\cap E)$ is the smallest number of open set required for a $r$-cover of $B(x,R)\cap E.$    
\end{definition}
\begin{example}\label{eg1}
The Assouad dimension of any finite subset of $\mathbb{R}^m$ is zero.    
\end{example}
Chen \cite{Chen} et al. explored the accessible values for the Assouad dimension of subsets. We now mention properties of the Assouad dimension in the following theorem, used in the article.
\begin{theorem}\cite[Lemma 2.4.1, Lemma 2.4.2]{Fraser}\label{th1}
\begin{enumerate}
    \item The Assouad dimension is monotone, finitely stable, stable under closure, and satisfies the open set property. 
    \item Let $E$ be a nonempty subset of $\mathbb{R}^m,$ and $T:\mathbb{R}^m\rightarrow\mathbb{R}^m$ be a bi-Lipschitz map. Then, $\dim_\mathcal{A}T(E)=\dim_\mathcal{A}E.$
\end{enumerate}   
\end{theorem}
Let $\mathcal{P}(\mathbb{R}^m)$ be the set of Borel probability measures on $\mathbb{R}^m$ with the metric defined as
$$d_L(\mu,\nu)=\sup_{f\in \text{Lip}_1}\Big\{\Big|\int_{\mathbb{R}^m}fd\mu-\int_{\mathbb{R}^m}fd\nu\Big|\Big\},$$
where $\text{Lip}_1$ is a set of all real-valued Lipschitz functions on $\mathbb{R}^m$ with the Lipschitz constant less than or equal to $1$, see \cite{Bill,Parth}. Let $\mu,\nu\in\mathcal{P}(\mathbb{R}^m).$ Then, the convolution $\mu*\nu$ is given as 
    $$\mu*\nu(A)=\int_{\mathbb{R}^m}\mu(A-x) d\nu(x),\hspace{.4cm} \mbox{for all Borel set } A \in \mathbb{R}^m.$$
      Equivalently, $\mu*\nu$ is the unique Borel probability measure satisfying 
      \begin{equation*}
\int_{\mathbb{R}^m}f(x)d(\mu*\nu)dx=\int_{\mathbb{R}^m}\int_{\mathbb{R}^m}f(x+y)d\mu(x)d\nu(y),~ \forall ~ f\in C_b(\mathbb{R}^m),\end{equation*}
 where $C_b(\mathbb{R}^m)$ denotes the set of all continuous and bounded functions on $\mathbb{R}^m$. 
We know that (i) $\mu*\nu(A)=\nu*\mu(A)$ and (ii)  for any $x_0\in\mathbb{R}^m,$ we have  $\mu*\delta_{x_0}(A)=\mu(A-x_0).$ Feng et al. \cite{Feng} studied the convolutions of equicontractive self-similar measures on $\mathbb{R}.$
\par
We now define $L^q$ dimension of a Borel probability measure with bounded support. This concept was initially introduced by R\'enyi \cite{Re}, and is one of the fundamental ingredients in the study of multifractal formalism; see \cite{JS}.
\begin{definition}\cite{Sh1} 
Let $q\in(1,\infty).$ If $\mu$ be a Borel probability measure on $\mathbb{R}$ with bounded support, then the $L^q$ spectrum of $\mu$ is defined as 
$$\tau(\mu,q)=\liminf_{k\to\infty}\frac{-\log\sum_{I\in D_k}\mu(I)^q}{k},$$ where $D_k=\{[j2^{-k}, (j+1)2^{-k}):j\in\mathbb{Z}\}$ is the level-$k$ dyadic partitions of $\mathbb{R}.$ The $L^q$ dimension of $\mu$ is given as $\dim_{L^q}(\mu)=\frac{\tau(\mu,q)}{q-1}.$\\
If $\mu$ is a Borel probability measure on $\mathbb{R}^m$ with bounded support, then the $L^q$ spectrum of $\mu$ is defined as 
$$\tau(\mu,q)=\liminf_{k\to\infty}\frac{-\log\sum_{I^1\times I^2\times\cdots \times I^m\in D_{k}^m}\mu(I^1\times I^2\times\cdots \times I^m)^q}{k},$$ where $D_k^m=\{[j2^{-k}, (j+1)2^{-k}):j\in\mathbb{Z}\}^m$ is the dyadic partitions of $\mathbb{R}^m.$ The $L^q$ dimension of $\mu$ is given as $\dim_{L^q}(\mu)=\frac{\tau(\mu,q)}{q-1}.$
\end{definition}
For entropy and $L^q$ dimension, researchers consider preferring logarithms on base $2,$ for normalizing the value of dimension. As $q\rightarrow1^+,$ the $L^q$ dimension tends to the Hausdorff dimension of the measure. Shmerkin \cite{Sh1} computed $L^q$ dimension of self-similar measure under ESC. Peres and Solomyak \cite{Peres} computed the $L^q$ dimension of self-conformal measures, without any separation assumptions. Several recent works on $L^q$ dimension can be studied in \cite{FB,O1,O2,O3}.
\begin{example}\label{eg2}
We now proceed to estimate the $L^q$ dimension of some well-known Borel probability measures,
to enhance clarity for new readers.
\begin{itemize}
    \item Let $x_0\in \mathbb{R}$ and $\delta_{x_0}$ be the Dirac measure supported on $x_0.$ Then,
        $$\tau(\delta_{x_0},q)=\liminf_{k\rightarrow\infty}\frac{-\log\sum_{I\in D_k}(\delta_{x_0}(I))^q}{k}=\liminf_{k\rightarrow\infty}\frac{-\log(1)^q}{k}=0.$$
        This gives $\dim_{L^q}(\delta_{x_0})=0.$
        \item Let $x_1,x_2,\ldots,x_n$ be distinct points in $\mathbb{R}$ and $c_1,c_2,\ldots , c_n>0.$ Then,
        $$\tau(c_1\delta_{x_1}+c_2\delta_{x_2}+\dots+c_n\delta_{x_n},q)=\liminf_{k\rightarrow\infty}\frac{-\log (c_1^q+c_2^q+\ldots +c_n^q)}{k}=0,$$
        and $\dim_{L^q}(c_1\delta_{x_1}+c_2\delta_{x_2}+\dots+c_n\delta_{x_n})=0.$
        \item Let $\mathcal{L}\big|_{[0,1]}$ be the normalized Lebesgue measure supported on $[0,1].$ Then, 
        \begin{align*}
            \tau(\mathcal{L}\big|_{[0,1]}, q)&=\liminf_{k\rightarrow\infty}\frac{-\log\sum_{I\in D_k}(\mathcal{L}\big|_{[0,1]}(I))^q}{k}\\
            &=\liminf_{k\rightarrow\infty}\frac{-\log(2^k-1)2^{-kq}}{k}\\
            &=\liminf_{k\rightarrow\infty}\frac{-\log(2^k-1)}{k}+\liminf_{k\rightarrow\infty}\frac{-\log2^{-kq}}{k}\\
            &=-\log2+q\log2=(q-1)\log2=q-1.
        \end{align*}
        Thus, $\dim_{L^q}(\mathcal{L}\big|_{[0,1]})=1$.
            \end{itemize}
\end{example}
Let $\Phi$ denote the similarity mapping on $\mathbb{R}^m,$ then 
$\Phi(x)=s\mathcal{O}x+d,$ where $s\in(-1,1)$ is the scaling ratio, $\mathcal{O}$ is the $m\times m$ orthogonal matrix, and $d\in \mathbb{R}^m$ is the translation vector. Hochman \cite{MHOCHMAN1} defined the distance between two similarities $\Phi$ and $\Phi_*$ in $\mathbb{R}^d$ as
\begin{align*}
    d(\Phi, \Phi_*):=|\log s-\log s_*|+\|\mathcal{O}-\mathcal{O}_*\|+\|d-d_*\|_2,
\end{align*}
where $\|.\|_2$ denotes the Euclidean norm in $\mathbb{R}^m$ and $\|.\|$ is the operator norm. Consider an IFS that consists of a finite number of similarity mappings on $\mathbb{R}^n$, i.e., $\mathcal{I}=\{\mathbb{R}^m; h_1,h_2,\ldots, h_N\}$ and 
\begin{equation}\label{delta1}
    \Delta_n=\inf\{d(h_i,h_j): i,j\in \{1,2,\ldots, N\}^n, i\ne j\}.
\end{equation}For $i=i_1i_2\ldots i_n\in \{1,2,\ldots N\}^n,$ we denote $h_i=h_{i_1}\circ h_{i_2}\circ\cdots \circ h_{i_n}.$ An IFS $\mathcal{I}$ is said to satisfy the exponential separation condition (ESC) if there exists a constant $\delta>0$ such that $\Delta_n>\delta^n$ for infinitely many $n\in \mathbb{N},$ where as $\mathcal{I}$ is said to satisfy the strong ESC if there exists a constant $\delta>0$ such that $\Delta_n>\delta^n,$ for all $n\in \mathbb{N},$ see \cite{BB1,MHOCHMAN1,Hochman2,Sh1}.\par
An example was presented by Garsia \cite{Garsia} as $\{\mathbb{R}:~h_1(x)=\gamma^{-1} x+1, h_2(x)=\gamma^{-1} x-1\},$ where $\gamma$ is the real root of the equation $x^3-x^2-2=0.$ Hochman \cite{MHOCHMAN1} shows that this IFS satisfies ESC, but not OSC. Meanwhile, super-exponential condensation means $\lim_{n\rightarrow\infty}\frac{1}{n}\log\Delta_n=-\infty.$ Bárány and Käenmäki \cite[Theorem 2.1]{BB1} showed the existence of an uncountable family of homogeneous self-similar sets of $\mathbb{R}$ that do not have total overlaps and satisfy super-exponential condensation. Baker \cite{Baker1} also independently showed the existence of such self-similar sets in $\mathbb{R}$ using the theory of continued fractions. Using the technique of Baker, Chen \cite{Chen1} provided some more examples in $\mathbb{R}.$ 


\section{Main Results}
\begin{lemma}\label{lm1}
    Let $A$ and $B$ be the matrices of order $k\times k$ and $m\times m,$ respectively. Then, 
    $$\Big\|\begin{bmatrix}
A & \textbf{0}_1 \\
\textbf{0}_2 & B
\end{bmatrix}\Big\|\ge\max\{\|A\|,\|B\|\},$$
where $\textbf{0}_1$ and $\textbf{0}_2$ are zero matrices of order $k\times m$ and $m\times k,$ respectively. Further, if $A$ and $B$ are orthogonal, so is $\begin{bmatrix}
A & \textbf{0}_1 \\
\textbf{0}_2 & B
\end{bmatrix}.$
\end{lemma}
\begin{proof}
Let $x=(x_1,x_2,\ldots,x_k)\in \mathbb{R}^k$ and $y=(y_1,y_2,\ldots,y_m)\in\mathbb{R}^m$ such that $(x,y)\in\mathbb{R}^{k+m}$ and $\|(x,y)\|_2\le1.$
Then,
\begin{align*}
   \Big\|\begin{bmatrix}
A & \textbf{0}_1 \\
\textbf{0}_2 & B
\end{bmatrix}\Big\|\ge  \Big\|\begin{bmatrix}
A & \textbf{0}_1 \\
\textbf{0}_2 & B
\end{bmatrix}\begin{pmatrix}
x \\
y
\end{pmatrix}\Big\|_2 =\Big\| \begin{pmatrix}
Ax \\
By
\end{pmatrix}\Big\|_2=\sqrt{\|Ax\|_2^2+\|By\|_2^2}.
\end{align*}
The previous inequality is also true, if we choose $(x,y)=(\textbf{0},y)\in\mathbb{R}^{k+m}$ with $\|(\textbf{0},y)\|_2\le1.$ So, 
\begin{align*}
   \Big\|\begin{bmatrix}
A & \textbf{0}_1 \\
\textbf{0}_2 & B
\end{bmatrix}\Big\|\ge \|By\|_2.
\end{align*}
 Taking the supremum for all $y\in \mathbb{R}^m$ with $\|y\|_2\le1,$ gives $\Big\|\begin{bmatrix}
A & \textbf{0}_1 \\
\textbf{0}_2 & B
\end{bmatrix}\Big\|\ge \|B\|.$ Similarly, by choosing $(x,\textbf{0})\in\mathbb{R}^{k+m}$ with $\|(x,\textbf{0})\|_2\le 1$ gives $\Big\|\begin{bmatrix}
A & \textbf{0}_1 \\
\textbf{0}_2 & B
\end{bmatrix}\Big\|\ge \|A\|,$ concluding the assertion.
\end{proof}

Consider the IFS $\mathcal{I}_1=\{\mathbb{R};h_1,h_2,\ldots,h_N\}$ and $\mathcal{I}_2=\{\mathbb{R};g_1,g_2,\ldots,g_N\}$ consists of similarity mappings on $\mathbb{R},$ i.e., $h_i(x)=c_ix+d_i$ and $g_i(x)=p_ix+q_i,$ where $c_i,p_i$ are similarity ratios in $(-1,1)$ and $d_i,q_i\in\mathbb{R}$ are translation vectors. The product of IFS $\mathcal{I}_1$ and $\mathcal{I}_2$ is given as $\mathcal{I}_1\times \mathcal{I}_2$ and is defined as 
$$\mathcal{I}_1\times \mathcal{I}_2:=\{\mathbb{R}^2;\ \psi_{ij}: i,j=1,2,\ldots, N\},$$
where $\psi_{ij}(x,y):=(h_i(x),g_j(y)).$ Clearly, for $\psi_{ij}$ to be simlitudes, $c_i=d_j=c,\ \forall \ i,j. $
Consider an example $\mathcal{I}_1=\Big\{[0,1]: h_1(x)=\frac{x}{3}, h_2(x)=\frac{x}{3}+\frac{1}{3}\Big\}$ and $\mathcal{I}_2=\Big\{[0,1]: g_1(x)=\frac{x}{2}, g_2(x)=\frac{x}{2}+\frac{1}{2}\Big\}.$ Then, $\psi_{11}(x,y)=\Big(\frac{x}{3},\frac{y}{2}\Big).$ Then, 
\begin{align*}
    \|\psi_{11}(x,y)-\psi_{11}(x_*,y_*)\|_2=\Big\|\Big(\frac{x}{3},\frac{y}{2}\Big)-\Big(\frac{x_*}{3},\frac{y_*}{2}\Big)\Big\|_2\le \frac{1}{2}\|(x,y)-(x_*,y_*)\|_2.
\end{align*}
Clearly, $\psi_{11}$ is a contraction on $\mathbb{R}^2,$ but not a similitude.

\begin{theorem}
    Let $\mathcal{I}_1=\{\mathbb{R}^k;h_1,h_2,\ldots,h_N\}$ and $\mathcal{I}_2=\{\mathbb{R}^m;g_1,g_2,\ldots,g_N\}$ be two IFSs such that $h_i(x)=cU_ix+a_i$ and $g_i(x)=cV_ix+b_i,$ where $c\in(-1,1), a_i\in\mathbb{R}^k, b_i\in \mathbb{R}^m, U_i$ and $V_i$ are $k\times k$ and $m\times m$ orthogonal matrices, respectively. If $\mathcal{I}_1$ and $\mathcal{I}_2$ satisfy strong ESC, then the product IFS $\mathcal{I}_1\times \mathcal{I}_2$ given as 
    $$\mathcal{I}_1\times \mathcal{I}_2=\{\mathbb{R}^{k+m};\psi_{ij}: i,j=1,2,\ldots,N\}, \ \psi_{ij}=(h_i,g_j)$$
    satisfies strong ESC.
\end{theorem}
\begin{proof}
Since IFSs $\mathcal{I}_1$ and $\mathcal{I}_2$ satisfy strong ESC, there exist constants $\delta_1$ and $\delta_2>0,$ respectively such that $$\Delta_n^{\mathcal{I}_1}>\delta_1^n ~~\text{ and }~~\Delta_n^{\mathcal{I}_2}>\delta_2^n,$$
for all $n\in\mathbb{N}.$ Notice that for $x\in \mathbb{R}^k$ and $y\in \mathbb{R}^m,$ one gets
\begin{align*}
    \psi_{ij}(x,y)=(h_i(x),g_j(y))=(cU_ix+a_i,cV_jy+b_j)=c\begin{bmatrix}
U_i & \textbf{0}_1 \\
\textbf{0}_2 & V_j
\end{bmatrix}\begin{pmatrix}
x \\
y 
\end{pmatrix}+\begin{pmatrix}
a_i \\
b_j
\end{pmatrix},
\end{align*}
where $\textbf{0}_1$ and $\textbf{0}_2$ are zero matrices of order $k\times m$ and $m\times k,$ respectively. In view of Lemma \ref{lm1}, one gets 
    \begin{align*}
        d(\psi_{ij},\psi_{i'j'})&=\Big\|\begin{bmatrix}
U_i & \textbf{0}_1 \\
\textbf{0}_2 & V_j
\end{bmatrix}-\begin{bmatrix}
U_{i'} & \textbf{0}_1 \\
\textbf{0}_2 & V_{j'}
\end{bmatrix}\Big\|+\Big\|\begin{pmatrix}
a_i \\
b_j
\end{pmatrix}-\begin{pmatrix}
a_{i'} \\
b_{j'}\end{pmatrix}\Big\|_2\\
&=\Big\|\begin{bmatrix}
U_i-U_{i'}& \textbf{0}_1 \\
\textbf{0}_2 & V_j-V_{j'}
\end{bmatrix}\Big\|+\Big\|\begin{pmatrix}
a_i-a_{i'}\\
b_j-b_{j'}
\end{pmatrix}\Big\|_2\\
&\ge \max\{ \|U_i-U_{i'}\|,\|V_j-V_{j'}\|\}+\sqrt{\|
a_i-a_{i'}\|_2^2+\|
b_j-b_{j'}\|_2^2}\\
&\ge \max\{ \|U_i-U_{i'}\|+\|
a_i-a_{i'}\|_2,\|V_j-V_{j'}\|+\|
b_j-b_{j'}\|_2\}\\
&=\max\{d(h_i,h_{i'}),d(g_j,g_{j'})\}\\
&\ge \max\{\Delta_n^{\mathcal{I}_1},\Delta_n^{\mathcal{I}_2}\}.
    \end{align*}
    Thus, \begin{align*}
    \Delta_n&=\inf\{d(\psi_{ij},\psi_{i'j'}): ij,i'j'\in \{1,2,\ldots,N\}^n, ij\ne i'j'\}\\
    &\ge\max\{\Delta_n^{\mathcal{I}_1},\Delta_n^{\mathcal{I}_2}\}
        >\min\{\delta_1^n,\delta_2^n\}=\delta_3^n, \end{align*}where $\delta_3=\min\{\delta_1,\delta_2\},$ for all $n\in\mathbb{N},$ assuring that $\mathcal{I}_1\times \mathcal{I}_2$ satisfies strong ESC.   
\end{proof}
\begin{remark}
    The previous theorem plays a considerable role in the context of extending the Hochman \cite[Theorem 1.4]{MHOCHMAN1} groundbreaking result to the product of IFSs. 
\end{remark}
Consider an IFS $\{\mathbb{R}^m;h_1,h_2,\ldots, h_N\}$ on $\mathbb{R}^m$ that consists of a finite number of affine mappings. Then, there exists a unique nonempty compact set $A\subset\mathbb{R}^m$ (termed as an invariant set) such that $A=\bigcup_{i=1}^Nh_i(A).$ Let $\text{co}(A)$ be the closed convex hull of $A.$ Define $\Delta_n^* $ as 
\begin{equation}\label{delta2}
    \Delta_n^*:=\inf\{\mathfrak{H}(h_i(\text{co}(A)), h_j(\text{co}(A))):i,j\in\{1,2,\ldots,N\}^n,i\ne j\}.
\end{equation}
Then, IFS $\{\mathbb{R}^m; h_1,h_2,\ldots, h_N\}$ is said to satisfy the \textbf{modified ESC} if there exists a $\delta>0$ such that $\Delta_n^*>\delta^n,$ for infinitely many $n\in \mathbb{N}.$ Similarly, we can define \textbf{modified strong ESC}.\par
It is worth emphasizing that our definition of ESC is natural from a geometric perspective and significantly convenient. Computability and applicability of ESC and modified ESC might be seen as one being more algebraic in nature and the other is more geometric in nature. It is noted that several researchers \cite{BB2,BB3,AR4} attempt to generalize and weaken the concept of ESC introduced by Hochman \cite{MHOCHMAN,MHOCHMAN1}. In a way, B\'ar\'any and M. Verma \cite{BB2} defined weak ESC for self-similar IFS, ESC for graph-directed self-similar IFS, ESC for common fixed point system. Above all that, the most fundamental advantage of the modified ESC is that it is defined not only for self-similar IFS but also extends to more general classes of IFS, including self-affine IFS and self-conformal IFS. 

\begin{example}
Some basic examples are presented for the reader's convenience.  Consider the middle third Cantor set $\mathcal{C}$ on $[0,1].$ Then, its IFS is given as $\{[0,1]:h_1(x)=\frac{x}{3}, h_2(x)=\frac{x}{3}+\frac{2}{3}\}.$ Clearly, $$\mathcal{C}=h_1(\mathcal{C})\cup h_2(\mathcal{C}).$$
Since the contraction ratios are the same for both maps in the IFS, only the translation vectors will play a role. For $i=11\ldots1\in \{1,2\}^n,$ we get $h_i(x)=\frac{x}{3^n},$ and for $i=22\ldots2\in \{1,2\}^n,$ we have $h_i(x)=\frac{1}{3^n}(x+2)+\frac{2}{3^{n-1}}+\cdots+\frac{2}{3}.$ Choose $i=11\ldots 1$ and $j=11\ldots12\in \{1,2\}^n.$ Then, $h_j(x)=\frac{x}{3^n}+\frac{2}{3^n}$ and   $$d(h_i,h_j)=\frac{2}{3^n}.$$ It follows that $$\Delta_n=\inf\{d(h_i,h_j):i,j\in\{1,2\}^n,i\ne j\}=\frac{2}{3^n}>\frac{1}{3^n}.$$ 
Clearly, $\delta=\frac{1}{3}.$ Further, the closed convex hull of the middle third Cantor set $\mathcal{C}$ is $[0,1].$ Then,
\begin{align*}
    \Delta_n^*&=\inf\{\mathfrak{H}(h_i(\text{co}(\mathcal{C})),h_j(\text{co}(\mathcal{C}))):i,j\in\{1,2\}^n, i\ne j\}\\
    &=\inf\{\mathfrak{H}(h_i([0,1]),h_j([0,1])):i,j\in\{1,2\}^n, i\ne j\}.
\end{align*}
Choose $i=11\ldots 1$ and $j=11\ldots12\in \{1,2\}^n.$ Then, $h_i([0,1])=\big[0,\frac{1}{3^n}\big]$ and $h_j([0,1])=\big[\frac{2}{3^n},\frac{1}{3^{n-1}}\big].$ Also $$\mathfrak{H}(h_i([0,1]),h_j([0,1]))=\mathfrak{H}\Big(\Big[0,\frac{1}{3^n}\Big],\Big[\frac{2}{3^n},\frac{1}{3^{n-1}}\Big]\Big)=\frac{2}{3^n},$$
and \begin{align*}
   \Delta_n^* =\inf\{\mathfrak{H}(h_i([0,1]),h_j([0,1])):i,j\in\{1,2\}^n, i\ne j\}=\frac{2}{3^n}>\frac{1}{3^n}.
\end{align*}
This clearly shows that the IFS satisfies both modified ESC and modified strong ESC.
The second example can be considered as Sierpi\'nski gasket on $\mathbb{R}^2.$ Its IFS is elaborated as $\{\mathbb{R}^2:~h_1,h_2,h_3\},$ where
\begin{align*}
    h_1\begin{pmatrix}
x \\
y 
\end{pmatrix}&=\frac{1}{2}\begin{bmatrix}
1 & 0 \\
0 & 1
\end{bmatrix}\begin{pmatrix}
x \\
y 
\end{pmatrix},\\
h_2\begin{pmatrix}
x \\
y 
\end{pmatrix}&=\frac{1}{2}\begin{bmatrix}
1 & 0 \\
0 & 1
\end{bmatrix}\begin{pmatrix}
x \\
y 
\end{pmatrix}+\begin{pmatrix}
\frac{1}{2} \\
0
\end{pmatrix},\\
h_3\begin{pmatrix}
x \\
y 
\end{pmatrix}&=\frac{1}{2}\begin{bmatrix}
1 & 0 \\
0 & 1
\end{bmatrix}\begin{pmatrix}
x \\
y 
\end{pmatrix}+\begin{pmatrix}
\frac{1}{4} \\
\frac{\sqrt{3}}{4}
\end{pmatrix}.
\end{align*}
Now we will try to estimate
$$\Delta_n=\inf\{d(h_i,h_j):i,j\in\{1,2,3\}^n,i\ne j\}.$$
Choose $i,j\in\{1,2,3\}^n,i\ne j$ as $i=11\ldots1$ and $j=11\ldots12.$ Then,
\begin{align*}
    h_i\begin{pmatrix}
x \\
y 
\end{pmatrix}=\frac{1}{2^n}\begin{bmatrix}
1 & 0 \\
0 & 1
\end{bmatrix}\begin{pmatrix}
x \\
y 
\end{pmatrix}~\text{ and }~
h_j\begin{pmatrix}
x \\
y 
\end{pmatrix}=\frac{1}{2^n}\begin{bmatrix}
1 & 0 \\
0 & 1
\end{bmatrix}\begin{pmatrix}
x \\
y 
\end{pmatrix}+\begin{pmatrix}
\frac{1}{2^n} \\
0
\end{pmatrix}.
\end{align*}
Thus, $d(h_i,h_j)=\frac{1}{2^n}$ and $\Delta_n=\frac{1}{2^n}.$ Now the closed convex hull of the Sierpi\'nski gasket is an equilateral triangle $\triangle$ with vertices $(0,0),(1,0,),(\frac{1}{2},\frac{\sqrt{3}}{2}),$ so
\begin{align*}
    \Delta_n^*&=\inf\{\mathfrak{H}(h_i(\triangle),h_j(\triangle)):i,j\in\{1,2,3\}^n, i\ne j\}.
\end{align*}
Choose $i=11\ldots 1$ and $j=11\ldots12\in \{1,2,3\}^n.$ Then, $h_i(\triangle)=$ an equilateral triangle $\triangle$ with vertices $(0,0),(\frac{1}{2^n},0),(\frac{1}{2^{n+1}},\frac{\sqrt{3}}{2^{n+1}})$ and $h_j(\triangle)=$ an equilateral triangle $\triangle$ with vertices $(\frac{1}{2^n},0),(\frac{1}{2^{n-1}},0),\\(\frac{3}{2^{n+1}},\frac{\sqrt{3}}{2^{n+1}}).$ Then, $\mathfrak{H}(h_i(\triangle),h_j(\triangle))=\frac{1}{2^n}$ and $\Delta_n^*=\frac{1}{2^n}.$ This clearly shows that the IFS satisfies both modified ESC and modified strong ESC.
  
\end{example}
\begin{remark}
 It is remarkable that $\Delta_n=0,$ for all $n\in\mathbb{N}$ if and only if $\Delta_n^*=0,$ for all $n\in\mathbb{N}.$ Let $\Delta_n^*=0,$ for all $n\in\mathbb{N}.$ Since the infimum is taken over a finite set, there exist $i,j\in\{1,2,\ldots,N\}^n$ with $ i\ne j$ such that
$$\mathfrak{H}(h_i(\text{co}(A)), h_j(\text{co}(A)))=0.$$
 It follows that $h_i=h_j$ on $co(A),$ that is, the maps $h_i$ and $h_j$ are same on the closed convex hull of the attractor. So, $d(h_i,h_j)=0,$ claiming that $\Delta_n=0,$ for all $n\in\mathbb{N}.$ On the same lines converse part is done. 
    
\end{remark}

\begin{theorem}\label{th3.4}
    Consider an IFS consisting of similarity transformations $(h_i(x)=c_iU_ix+a_i, i=1,2,\ldots,N)$ on $\mathbb{R}^m,$ then there exists a constant $k_*\in \mathbb{R}$ such that $\Delta_n^* \le k_*\Delta_n,$ for all $n\in\mathbb{N}.$
\end{theorem}
\begin{proof}
    Let $i=i_1i_2\ldots i_n\in\{1,2,\ldots,N\}^n.$  Then,
\begin{align*}
    h_i(x) &= h_{i_1} \circ h_{i_2} \circ \dots \circ h_{i_n}(x) \\
    &= h_{i_1} \circ h_{i_2} \circ \dots \circ h_{i_{n-1}}(c_{i_n} U_{i_n} x + a_{i_n}) \\
    &= h_{i_1} \circ h_{i_2} \circ \dots \circ h_{i_{n-2}}( c_{i_{n-1}}c_{i_n} U_{i_{n-1}}U_{i_n}  x + c_{i_{n-1}}U_{i_{n-1}}a_{i_n}+a_{i_{n-1}})\\
&= (c_{i_1} \dots  c_{i_{n-1}} c_{i_n})(U_{i_1} \dots U_{i_n} x) +
      (c_{i_1}\ldots c_{i_{n-1}}U_{i_1}\ldots U_{i_{n-1}}a_{i_n}+\\
        &\quad c_{i_1}\ldots c_{i_{n-2}}U_{i_1}\ldots U_{i_{n-2}}a_{i_{n-1}}+\cdots+c_{i_1}U_{i_1}a_{i_2}+a_{i_1}).
\end{align*}
    Let $c_i =  c_{i_1} \dots c_{i_n},~U_i=U_{i_1} \dots U_{i_n},$ and 
    $d_i=c_{i_1}\ldots c_{i_{n-1}}U_{i_1}\ldots U_{i_{n-1}}a_{i_n}+c_{i_1}\ldots c_{i_{n-2}}U_{i_1}\ldots U_{i_{n-2}}a_{i_{n-1}}\\+\cdots+c_{i_1}U_{i_1}a_{i_2}+a_{i_1}.$
It follows that $h_i(x) = c_i U_i x + d_i.$ For $i,j\in \{1,2,\ldots,N\}^n,$ we have
\begin{align*}
   &d(h_i,h_j)=\left|\log c_i-\log c_j\right|+\left\|U_i-U_j\right\| 
+\left\|d_i-d_j\right\|_2, ~~\text{ and} \\
\mathfrak{H}(h_i(\text{co}(A)), &h_j(\text{co}(A)))=\max\Big\{\sup_{x\in \text{co}(A)}\inf_{y\in \text{co}(A)}\|h_i(x)-h_j(y)\|_2,\sup_{y\in \text{co}(A)}\inf_{x\in \text{co}(A)}\|h_i(x)-h_j(y)\|_2\Big\}.
\end{align*}
    Now,
    \begin{align*}
     \inf_{y\in \text{co}(A)}\|h_i(x)-h_j(y)\|_2   \le \|h_i(x)-h_j(x)\|_2.
    \end{align*}
    Taking the supremum of both sides, we get
    \begin{align*}
        \sup_{x\in \text{co}(A)}\inf_{y\in \text{co}(A)}\|h_i(x)-h_j(y)\|_2   \le \sup_{x\in \text{co}(A)}\|h_i(x)-h_j(x)\|_2.
    \end{align*}
    Similarly,
    \begin{align*}
\sup_{y\in \text{co}(A)}\inf_{x\in \text{co}(A)}\|h_i(x)-h_j(y)\|_2\le\sup_{y\in \text{co}(A)}\|h_i(y)-h_j(y)\|_2.
    \end{align*}
    This gives 
    \begin{align}\label{eq38}
       \mathfrak{H}(h_i(\text{co}(A)), h_j(\text{co}(A))) \le\sup_{x\in \text{co}(A)}\|h_i(x)-h_j(x)\|_2.
    \end{align}
Consider
\begin{align*}
 \|h_i(x)-h_j(x)\|_2 & =\|c_i U_i x+d_i-(c_j U_j x+d_j)\|_2\\
& \leq \|( c_i U_i-c_j U_j) x\|_2+\| d_i-d_j \|_2\\
&\leq\|c_i U_i-c_j U_j\| \cdot\|x\|_2  +\| d_i-d_j \|_2\\&
\leq\|c_i U_i-c_j U_j\| \cdot k+\| d_i-d_j \|_2,    
\end{align*}
where $k=\max \{\|x\|_2: x \in \text{co}(A)\}.$ Let $k_*=\max \{k, 1\}.$
\begin{align*}
\|h_{i}(x)-h_{j}(x)\|_2 & \leq k_*(\|c_{i} U_{i}-c_{j} U_{j}\|+\|d_{i}-d_{j}\|_2) \\
& \leq k_*(|c_{i}-c_{j}|\| U_{i}\|+| c_{j}|\|U_{i}-U_{j}\|+\|d_{i}-d_{j}\|_2) \\
& \leq  k_*(| c_{i}-c_{j}|+\|U_{i}-U_{j}\|+\|d_{i}-d_{j}\|_2),   
\end{align*}
as $c_j<1, \ \forall ~j\in \{1,2,\ldots,N\}^n. $ Using mean value theorem on interval $[c_i,c_j]$ if $c_i<c_j,$ (otherwise $[c_j,c_i],$ we get $|c_{i}-c_j| <|\log c_{i}-\log c_{j}|.$ So, 
\begin{align*}
\|h_i(x)-h_j(x)\|_2 &\leq k_*(|\log c_i-\log c_j| +\|U_i-U_j\| +\|d_i-d_j\|_2) \\
& \leq k_* \ d\left(h_i, h_j\right).   
\end{align*}
Consequently, $\sup_{x\in \text{co}(A)}\|h_i(x)-h_j(x)\|_2 \leq k_* \ d(h_i, h_j).$ Using equation \eqref{eq38}, we get


$$  \mathfrak{H}(h_i(\text{co}(A)), h_j(\text{co}(A))) \leq k_* \ d(h_i, h_j),
~~ \forall  ~i, j. $$ 
Since the constant $k_*$ is independent of $n$, taking infimum on both sides, we have $$ \Delta_n^* \le k_*\Delta_n.$$
Clearly, this holds for all $n\in\mathbb{N}.$ This completes the proof.


\end{proof}
    
\begin{proposition}\label{prop1}
 Consider a homogeneous IFS consisting of similarity transformations $(h_i(x)=cx+d_i, i=1,2,\ldots,N)$ on $\mathbb{R},$ then the above two definitions (mentioned in equations \eqref{delta1} and \eqref{delta2}) of ESC coincide, that is, $\Delta_n = \Delta_n^*,$ for all $n\in\mathbb{N}.$
 \end{proposition}
  \begin{proof}
    Consider $i,j\in\{1,2,\ldots,N\}^n,$ we get
\begin{align*}
    d(h_i,h_j)=|d_i-d_j|.
\end{align*}
In $\mathbb{R},$ the closed convex hull of attractor $A$ is a compact interval $I=[a,b]$ say, then
\begin{align*}
     \mathfrak{H}(h_i(\text{co}(A), h_j(\text{co}(A)))&=\mathfrak{H}(h_i(I), h_j(I))\\
     &=\mathfrak{H}([ac+d_i,bc+d_i], [ac+d_j,bc+d_j])\\
     &=\max\{|ac+d_i-(ac+d_j)|,|bc+d_i-(bc+d_j)|\}\\
     &=|d_i-d_j|.
\end{align*}
Thus, in this case
\begin{align*}
     \mathfrak{H}(h_i(\text{co}(A)), h_j(\text{co}(A)))=|d_i-d_j|= d(h_i,h_j).
\end{align*}
Taking the infimum on both sides over $i,j\in \{1,2,\ldots,N\}^n,$ we get $\Delta_n^*=\Delta_n,$ for all $n\in\mathbb{N},$ concluding the proof. 
     \end{proof}
     \begin{note}
       If homogeneous IFS on $\mathbb{R}$ satisfies ESC, there exists a constant $\delta>0$ such that $\Delta_n>\delta^n,$ for infinitely many $n\in\mathbb{N}.$ This gives $\Delta_n^*>\delta^n,$ for infinitely many $n\in\mathbb{N}.$ Thus, the homogeneous IFS satisfies modified ESC.  
     \end{note}
The next remark will give some hints on the ESC for the distinct contraction ratios.
  
  \begin{remark}
  
  Consider an IFS on $\mathbb{R}$ with similarity mappings $(h_i(x)=c_ix+d_i,~ i=1,2,\ldots,N)$ with attractor $A.$ The closed convex hull of $A$ is a compact interval $I=[a,b],$ say. Consider the case $0\le a <b,~c_i-c_j>0$ and $d_i-d_j\ge0,$ for all $i,j\in \{1,2,\ldots,N\}^n.$ 
Suppose $i,j\in\{1,2,\ldots,N\}^n,$ we get
\begin{align*}
    d(h_i,h_j)=|\log c_i-\log c_j|+|d_i-d_j|.
\end{align*}
The $\log$ function will satisfy $$\frac{|\log c_i-\log c_j|}{|c_i-c_j|}\le \frac{1}{\min_{i\in\{1,2,\ldots,N\}^n}|c_i|}=\frac{1}{(\min_{i\in \{1,2,\ldots,N\}}|c_i|)^n}.$$
Let $k_n=\frac{1}{(\min_{i\in \{1,2,\ldots,N\}}|c_i|)^n}.$ So, 
\begin{align*}
d(h_i,h_j)&\le k_n|c_i- c_j|+|d_i-d_j|\\
 &\le k_n\{(c_i- c_j)+(d_i-d_j)\}.
\end{align*}
Again, notice that
\begin{align*}
     \mathfrak{H}(h_i(\text{co}(A)), h_j(\text{co}(A)))&=\mathfrak{H}(h_i(I), h_j(I))\\
     &=\mathfrak{H}([ac_i+d_i,bc_i+d_i], [ac_j+d_j,bc_j+d_j])\\
     &=\max\{|ac_i+d_i-(ac_j+d_j)|,|bc_i+d_i-(bc_j+d_j)|\}\\
     &=\max\{|a(c_i-c_j)+(d_i-d_j)|,|b(c_i-c_j)+(d_i-d_j)|\}.
\end{align*}
  Without loss of generality, assume that
    \begin{align*}
       \mathfrak{H}(h_i(\text{co}(A)), h_j(\text{co}(A)))&=b(c_i-c_j)+(d_i-d_j)\\
       &\ge \min\{b,1\}\{(c_i-c_j)+(d_i-d_j)\},
    \end{align*}
    that is, $$\frac{1}{\min\{b,1\}}\mathfrak{H}(h_i(\text{co}(A)), h_j(\text{co}(A)))\ge \{(c_i-c_j)+(d_i-d_j)\}\ge \frac{1}{k_n} d(h_i,h_j).$$
    Taking infimum over $i,j\in \{1,2,\ldots,N\}^n$ both sides, we get $\frac{1}{\min\{b,1\}}\Delta_n^*\ge  \frac{1}{k_n}\Delta_n.$ If the IFS satisfies ESC, there exists a constant $\delta>0$ such that $\Delta_n>\delta^n,$ for infinitely many $n\in\mathbb{N}.$ This gives $\Delta_n^*>\frac{\min\{b,1\}}{k_n}\delta^n=\min\{b,1\}\Big(\min_{i\in\{1,2,\ldots,N\}}|c_i|\delta\Big)^n,$ for infinitely many $n\in\mathbb{N},$ claiming that the IFS satisfies modified ESC. 

    \end{remark}

  In the upcoming theorems, several dense subsets of $\chi(\mathbb{R}^m)$ are explored, which will be helpful in the dimension-preserving approximation of sets.
\begin{theorem}
    For each $\alpha\in[0,m],$ the set $A_\alpha=\{E\in \chi(\mathbb
{R}^m): \dim_\mathcal{A}E=\alpha\}$ is dense in $\chi(\mathbb
    {R}^m).$ 
\end{theorem}
 \begin{proof}
     Let $\epsilon>0$ and $E\in \chi(\mathbb{R}^m).$ Since $E$ is compact subset of $\mathbb{R}^m,$ there exist $e_1,e_2,\ldots,e_n\in E$ such that $E\subset \bigcup_{j=1}^nB(e_j,\epsilon),$ where $B(e_j,\epsilon)$ is the open set centered at $e_j$ and radius $\epsilon.$ Subsequently, consider $F=\{e_1,e_2,\dots,e_n\}\in \chi(\mathbb{R}^m).$ In view of definition of Hausdorff metric and Example \ref{eg1}, we get 
     $$\mathfrak{H}(E,F)<\epsilon~~\text{ and }~~\dim_{\mathcal{A}}F=0.$$
     This shows that the assertion is true in the case of $\alpha=0.$ Again, assume $\alpha\in(0,m].$ Since the Assouad dimension satisfies the open set property (see Theorem \ref{th1}), $\dim_{\mathcal{A}}B(e_j,\epsilon)=m.$ Now using \cite[Theorem 1]{Chen}, there exists a set $F_{e_1}\subset B(e_1,\epsilon)$ such that $\dim_{\mathcal{A}}F_{e_1}=\alpha.$ Define $F'=F\cup F_{e_1}.$ In view of the finite stability of the Assouad dimension, we get
     $$\dim_{\mathcal{A}}F'=\max\{\dim_{\mathcal{A}}F,\dim_{\mathcal{A}}F_{e_1}\}=\alpha,$$
      and $\mathfrak{H}(E,F')<\epsilon,$ conclude the claim.
 \end{proof} 
\begin{theorem}
    For each $\alpha\in[0,m],$ the set $H_\alpha=\{E\in \chi(\mathbb
    R^m): \dim_\mathcal{H}E=\alpha\}$ is dense in $\chi(\mathbb
    R^m).$ 
\end{theorem}
\begin{proof}
    Proof follows on similar arguments as in the previous theorem, hence we avoid the details.
\end{proof}
Note that for any $E\subset\mathbb{R}^m,$ we have $\dim_\mathcal{H}E\le\dim_\mathcal{A}E.$ Thus, our next result is worth mentioning.
\begin{theorem}
    The set $\mathcal{D}=\{E\in\chi(\mathbb{R}^m):\dim_\mathcal{H}E\ne\dim_\mathcal{A}E\}$ is dense in $\chi(\mathbb{R}^m).$ 
\end{theorem}
\begin{proof}
       In case of $m=1,$ using \cite[Theorem 2.1.1]{Fraser}, one can clearly show that $\mathcal{D}$ is a non-empty subset of $\chi(\mathbb{R}^m).$ The Bedford-McMullen carpet (denoted as $\mathcal{B}\subset\mathbb{R}^2$) has the Hausdorff dimension strictly less than the Assouad dimension (see \cite[Theorem 2.1]{Fraser1}) is an example in case of $m=2.$ For a higher dimension $(m>2)$, consider the set $\mathcal{B}\times\{(\underset{m-2\text{ times}}{\underbrace{0,0,\dots,0}} )\}=\{(x,y,0,0,\ldots,0)\in\mathbb{R}^m:(x,y)\in\mathcal{B}\},$  concluding the set $\mathcal{D}$ is non-empty. Further, assume $E\in\chi(\mathbb{R}^m)$ and $\epsilon>0.$ Then, there exist $e_1,e_2,\ldots,e_n\in E$ such that $E\subset\cup_{j=1}^nB(e_j,\epsilon).$ Observe that the Hausdorff and the Assouad dimensions are invariant under scaling and translation (see \cite{Fal,Fraser}). Using translation and scaling, we get $\mathcal{B}\times\{(0,0,\ldots,0)\}\subset B(e_1,\epsilon).$ Choose $F=\{e_1,e_2,\ldots,e_n\}\cup (\mathcal{B}\times\{(0,0,\ldots,0)\}).$ Since the Assouad and the Hausdorff dimensions satisfy finite stability, we get
       \begin{align*}
           &\dim_{\mathcal{A}}F=\dim_{\mathcal{A}}(\mathcal{B}\times\{(0,0,\ldots,0)\}),\\
           &\dim_{\mathcal{H}}F=\dim_{\mathcal{H}}(\mathcal{B}\times\{(0,0,\ldots,0)\}).
       \end{align*}
       Clearly, $\dim_{\mathcal{A}}F\ne\dim_{\mathcal{H}}F$ and $\mathfrak{H}(E,F)<\epsilon,$ concluding the assertion.
\end{proof}
\noindent Before mentioning our upcoming theorem, we define the following for an IFS $\mathcal{I}=\{\mathbb{R}^m;h_1,h_2,\ldots, h_N\}$  with $h_i(x)=c_iU_ix+a_i,$ where $0<c_i<1$ and $U_i$ is $m\times m$ orthogonal matrix, $a_i\in \mathbb{R}^m.$
\begin{itemize}
    \item An IFS $\mathcal{I}$ is said to be affinely irreducible if the only trivial affine subspaces are simultaneously preserved by all $h_i.$
    \item A linear subspace $V\subseteq\mathbb{R}^m$ is said to be $D\mathcal{I}$-invariant if it is invariant under the orthogonal parts, i.e., $U_i(V)\subseteq V,$ for all $i.$
\end{itemize}
We now mention the following theorem, which is a conclusion of \cite[Theorem 1.4]{MHOCHMAN1}.
\begin{theorem}\label{th2}
Let $A\subset\mathbb{R}^m$ be the  attractor of an affinely irreducible IFS given as $\mathcal{I}=\{\mathbb{R}^m;h_1,h_2,\ldots, h_N\}$  with $h_i(x)=c_iU_ix+a_i,$ where $0<c_i<1$ and $U_i$ is $m\times m$ orthogonal matrix, $a_i\in \mathbb{R}^m.$ Suppose that $\mathcal{I}$ satisfies ESC     
    and there does not exist a non-trivial $D\mathcal{I}$-invariant linear subspace $V\subseteq\mathbb{R}^m$ and $e\in A$ with $\dim_{\mathcal{H}}(A\cap (V+e))=\dim_{\mathcal{H}} V.$ Then, $\dim_{\mathcal{H}} A=\min\{m,\dim_{S}A\},$ where $\dim_SA$ is the similarity dimension of $A.$
\end{theorem}
Let $\Lambda=\bigcup_{n\in\mathbb{N}}\{1,2,\ldots,N\}^n$ be the collection of all finite alphabets on $\{1,2,\ldots,N\}.$

\begin{theorem}\label{th310}
Let $\mathcal{I}$ be an IFS consisting of a finite number of similarity transformations (say, $h_i(x)=c_iU_ix+a_i,~ i=1,2,\ldots,N$) in $\mathbb{R}^m$ with attractor as $A.$ If there exists an alphabet $i \in \Lambda$ such that the IFS $\mathcal{I}_n=\{\mathbb{R}^m;~ h_{ji}: j \in\{1,2,\ldots,N\}^n \}$ is affinely irreducible with attractor $A_n$ satisfies ESC and there does not exist a non-trivial $D\mathcal{I}_n$-invariant linear subspace $V\subseteq\mathbb{R}^m$ and $e\in A_n$ with $\dim_{\mathcal{H}}(A_n\cap (V+e))=\dim_{\mathcal{H}} V.$ Then, $\dim_{\mathcal{H}} A=\min\{m,\dim_{S}A\},$ where $\dim_SA$ is the similarity dimension of $A.$

\end{theorem}
\begin{proof}
The IFS $\mathcal{I}_n$ satisfies the assumptions of Theorem \ref{th2}, so $\dim_{\mathcal{H}}A_n=\min\{m,\dim_SA_n\},$ where $\sum_{j\in\{1,2,\ldots,N\}^n}c_{ji}^{\dim_SA_n}=1.$ Using the fact $A_n\subseteq A,$ we get $\dim_\mathcal{H}A_n\le\dim_\mathcal{H}A.$ In the first case
$\dim_\mathcal{H}A_n=m,$ then $\dim_\mathcal{H}A=m,$ we are done. In another case, consider $\dim_\mathcal{H}A_n=\dim_SA_n.$ On the contrary, suppose that $\dim_\mathcal{H}A<\dim_SA.$ Then,
\begin{align*}
  &c_i^{-\dim_SA_n}=\sum_{j\in\{1,2,\ldots,N\}^n}c_{j}^{\dim_SA_n} \ge \sum_{j\in\{1,2,\ldots,N\}^n}c_{j}^{\dim_\mathcal{H}A}=\\&\sum_{j\in\{1,2,\ldots,N\}^n}c_{j}^{(\dim_\mathcal{H}A-\dim_SA)}c_j^{\dim_SA}\ge \sum_{j\in\{1,2,\ldots,N\}^n}c_{\max}^{n(\dim_\mathcal{H}A-\dim_SA)}c_j^{\dim_SA}\\
  &=c_{\max}^{n(\dim_\mathcal{H}A-\dim_SA)}.
\end{align*}
As $n\rightarrow\infty,$ we get a contradiction. Thus, $ \dim_\mathcal{H}A=\dim_SA,$ concluding the proof. 
\end{proof}
\begin{remark}
  It is straightforward to observe that if an IFS satisfies ESC (SSC or OSC or SOSC) at the first level, then all subcollections of the IFS at $n$-th level also satisfy ESC (SSC or OSC or SOSC). But the converse is not true in general. That is, if there exists a subcollection of the IFS of the $n$-th level that satisfies ESC (SSC or OSC or SOSC), then the original IFS does not necessarily satisfy ESC (SSC or OSC or SOSC). As an example, consider an IFS on $\mathbb{R}$ with three mappings as follows:
  $$h_1(x)=\frac{x}{4},\quad h_2(x)=\frac{x}{4}+\frac{9}{16}, \quad h_3(x)=\frac{x}{4}+\frac{3}{4}.$$ It is noted that the IFS does not satisfy SSC, but there exists a subcollection of $2$-nd level IFS satisfying SSC (see \cite{MHOCHMAN,VP1}). In view of this, the previous theorem is a weaker version of the Hochman result \cite[Theorem 1.4]{MHOCHMAN1}.    
\end{remark}

The upcoming lemmas are elementary observations on the behaviour of the $L^q$ dimension, and their proofs are accommodated for young readers. It is remarkable here that Feng et al. \cite[Lemma 2.2(ii)]{Feng} proved one-sided inequality for general Borel probability measures by defining $L^q$ dimension in terms of integration.
\begin{lemma}\label{convothm}
    Let $\mathcal{S}(\mathbb{R}^m)=\big\{\sum_{j=1}^nc_j\delta_{x_j}: x_j\in\mathbb{R}^m,c_j>0,\sum_{j=1}^nc_j=1,n\in\mathbb{N}\big\}$ and $\mu$ be a Borel probability measure with finite support. Then, for a fixed $\omega\in\mathcal{S}(\mathbb{R}^m),$  $$\dim_{L^q}(\mu*\omega)=\dim_{L^q}(\mu),$$
    where $\mu*\omega$ is the convolution of $\mu$ and $\omega.$
\end{lemma}
\begin{proof}
We first observe our claim for the Dirac measures. For this, consider $x\in\mathbb{R}$ with $x$ as dyadic end point, i.e., $x=\frac{j_0}{2^k}.$ Then,
\begin{align*}
    I-x=\Big[\frac{j}{2^k}-\frac{j_0}{2^k},\frac{j+1}{2^k}-\frac{j_0}{2^k}\Big)=\Big[\frac{j-j_0}{2^k},\frac{j+1-j_0}{2^k}\Big).
\end{align*}
In this case, $I-x$ is another subinterval in $D_k.$ So, $\dim_{L^q}(\mu*\delta_x)=\dim_{L^q}(\mu).$
In second case observe that for any $x\in\mathbb{R},$ consider
 \begin{align*}
     I-x=\Big[\frac{j}{2^k}-x,\frac{j+1}{2^k}-x\Big).
 \end{align*}
 Note that the length of $I-x$ is the same as the length of $I,$ equal to $\frac{1}{2^k}.$ So. $I-x$ must be contained in two consecutive intervals (say, $I'$ and $I''$) of $D_k$.
 \begin{align*}
     \sum_{I\in D_k}\Big(\mu*\delta_{x}(I)\Big)^q=\sum_{I\in D_k}\Big(\mu(I-x)\Big)^q\le\sum_{I',I''\in D_k}\Big(\mu(I')+\mu(I'')\Big)^q,
 \end{align*}\\
 where $I-x\subset I'\cup I''.$
 \begin{align*}
     \sum_{I\in D_k}\Big(\mu*\delta_{x}(I)\Big)^q\le\sum_{I',I''\in D_k}2^{q-1}\Big(\mu(I')^q+\mu(I'')^q\Big)\le 2^q\sum_{I\in D_k}\mu(I)^q.
 \end{align*}
 Consequently, 
 \begin{align*}
     -\log\sum_{I\in D_k}\Big(\mu*\delta_{x}(I)\Big)^q\ge -\log \Big(2^q\sum_{I\in D_k}\mu(I)^q\Big)=-\log(2^q)-\log\Big(\sum_{I\in D_k}\mu(I)^q\Big).
 \end{align*}
 Dividing by $k$ and taking the limit infimum, we get
 $$\tau(\mu*\delta_x,q)\ge\tau(\mu,q)~\text{ and }~\dim_{L^q}(\mu*\delta_x)\ge\dim_{L^q}(\mu).$$
 Again, assume $I$ and $J$ are distinct intervals in $D_k.$ Then, for any $x\in\mathbb{R},~ I-x$ and $J-x$ are disjoint because for $z\in (I-x)\cap (J-x)$, it follows $z=y-x$ and $z=y'-x$ for some $y\in I$ and $y'\in J.$ So $x+z\in I\cap J,$ which is a contradiction. Thus, the collection $\{I-x: I\in D_k\}$ contains mutually disjoint intervals. For any $I\in D_k,$ we have
 $$I\subset (I'-x)\cup(I''-x) ~~\text{ for some } I', I'' \in D_k.$$
 Following the same lines as in the previous one, we get $\dim_{L^q}(\mu*\delta_x)\le\dim_{L^q}(\mu).$
 Let $\omega\in \mathcal{S}(\mathbb{R}^m)$ be of the form $\omega=\sum_{j=1}^nc_j\delta_{x_j},$ for some $x_j\in\mathbb{R}^m, c_j>0$ and $\sum_{j=1}^nc_j=1.$ Then, for $I\in D_k,$ we have
 \begin{align*}
\Big(\Big(\mu*\sum_{j=1}^nc_j\delta_{x_j}\Big)(I)\Big)^q&=\Big(\Big(\sum_{j=1}^nc_j\mu*\delta_{x_j}\Big)(I)\Big)^q\\
&=\Big(\sum_{j=1}^nc_j^{q'}\Big)^\frac{q}{q'}\Big(\sum_{j=1}^n(\mu*\delta_{x_j}(I))^q\Big),
 \end{align*}
 where $q'>1$ such that $\frac{1}{q'}+\frac{1}{q}=1.$ Let $M=(\sum_{j=1}^nc_j^{q'}\Big)^\frac{q}{q'}.$ Thus, 
 \begin{align*}
     \sum_{I\in D_k}\Big(\Big(\mu*\sum_{j=1}^nc_j\delta_{x_j}\Big)(I)\Big)^q&=M\sum_{I\in D_k}\Big(\sum_{j=1}^n(\mu*\delta_{x_j}(I))^q\Big)\\
     &=M\sum_{j=1}^n\Big(\sum_{I\in D_k}(\mu*\delta_{x_j}(I))^q\Big)\\
     &\le 2^qMn\Big(\sum_{I\in D_k}\mu(I)^q\Big).
 \end{align*}
 This gives $\tau(\mu*\sum_{j=1}^nc_j\delta_{x_j},q)\ge\tau(\mu,q).$ Again,
 \small{\begin{align*}
     & \sum_{I\in D_k}\mu
     (I)^q=\sum_{I\in D_k}\sum_{j=1}^nc_j\mu
     (I)^q=\sum_{I\in D_k}c_1\mu
     (I)^q+\cdots+\sum_{I\in D_k}c_n\mu
     (I)^q\le  \\
     &2^q\sum_{I\in D_k}c_1\Big(\mu*\delta_{x_1}(I)\Big)^q+\cdots+2^q\sum_{I\in D_k}c_n\Big(\mu*\delta_{x_n}(I)\Big)^q\le2^q\sum_{I\in D_k}\Big(c_1\mu*\delta_{x_1}(I)\Big)^q\\
     &+\cdots+2^q\sum_{I\in D_k}\Big(c_n\mu*\delta_{x_n}(I)\Big)^q\le2^q\sum_{I\in D_k}\Big(\sum_{j=1}^nc_j\mu*\delta_{x_j}(I)\Big)^q=2^q\sum_{I\in D_k}\Big(\mu*\sum_{j=1}^nc_j\delta_{x_j}(I)\Big)^q,
 \end{align*}}
 
 concluding that $\tau(\mu*\sum_{j=1}^nc_j\delta_{x_j},q)\le\tau(\mu,q).$ This completes the proof.
\end{proof}
\begin{lemma}\label{1179}
    Let $\mu$ be a Borel probability measure on $\mathbb{R}^m$ and $\beta$ be a non-zero real number. Then, $\dim_{L^q}(\mu)=\dim_{L^q}(\nu),$ where $\nu(D)=\mu(\beta D),$ for all Borel set $D.$  
\end{lemma}
\begin{proof}
The result will be demonstrated in $\mathbb{R}$ and following a similar approach, we extend it to $\mathbb{R}^m.$
    Firstly, suppose $\beta$ as a dyadic end point, that is, $\beta=\frac{j_0}{2^s},$ with $s>0.$ Then, $\beta I=\Big[\frac{jj_0}{2^{k+s}},\frac{jj_0+j_0}{2^{k+s}}\Big),$ and its length is equal to $\frac{j_0}{2^{k+s}}.$ Notice that $$\beta I=\Big[\frac{jj_0}{2^{k+s}},\frac{jj_0+1}{2^{k+s}}\Big)\cup\Big[\frac{jj_0+1}{2^{k+s}},\frac{jj_0+2}{2^{k+s}}\Big)\cup\cdots \cup \Big[\frac{jj_0+j_0-1}{2^{k+s}},\frac{jj_0+j_0}{2^{k+s}}\Big).$$  
Consider $I_1=\Big[\frac{jj_0}{2^{k+s}},\frac{jj_0+1}{2^{k+s}}\Big),\ldots, I_{j_0}=\Big[\frac{jj_0+j_0-1}{2^{k+s}},\frac{jj_0+j_0}{2^{k+s}}\Big),$ so $\mu(\beta I)=\mu(I_1)+\cdots+\mu(I_{j_0}).$ Thus, 
\begin{align*}
    \sum_{I\in D_k}\nu(I)^q&=\sum_{I\in D_k}\mu(\beta I)^q=\sum_{I_1,\ldots,I_{j_0}\in D_{k+s}}\Big(\mu(I_1)+\cdots+\mu(I_{j_0})\Big)^q\\&\le C\Big\{\sum_{I_1\in D_{k+s}}\mu(I_1)^q+\cdots+\sum_{I_{j_0}\in D_{k+s}}\mu(I_{j_0})^q\Big\}=j_0C\sum_{I\in D_{k+s}}\mu(I)^q,
\end{align*}
where $C$ is a suitable constant. It follows that $$\frac{-\log\sum_{I\in D_k}\nu(I)^q}{k}\ge \frac{-\log\Big(j_02^{q-1}\sum_{I\in D_{k+s}}\mu(I)^q\Big)}{k+s}.$$ Applying $\liminf_{k\rightarrow\infty}$ on both sides, we get $\dim_{L^q}(\nu)\ge\dim_{L^q}(\mu).$ For reverse inequality, we have $$I=\Big[\frac{j}{2^k},\frac{j+1}{2^k}\Big)=\Big[\frac{\beta j2^s}{j_02^k},\frac{\beta(j+1)2^s}{j_02^k}\Big)=\beta\Big[\frac{ j}{j_02^{k-s}},\frac{(j+1)}{j_02^{k-s}}\Big).$$ 
Suppose that the interval $\Big[\frac{ j}{j_02^{k-s}},\frac{(j+1)}{j_02^{k-s}}\Big)$
 is contained in intervals $I_1,I_2,\ldots,I_l$ dyadic intervals of level $k-s.$ Then,
 $$I\subset \beta I_1\cup\beta I_2\cup\ldots\cup\beta I_l,$$
 and $\mu(I)\le\mu(\beta I_1)+\mu(\beta I_2)+\cdots+\mu(\beta I_l)=\nu(I_1)+\nu(I_2)+\cdots+\nu(I_l).$ Subsequently,
 \begin{align*}
 &\sum_{I\in D_k}\mu(I)^q\le\sum_{I_1,I_2,\ldots,I_l\in D_{k-s}}\Big(\nu(I_1)+\nu(I_2)+\cdots+\nu(I_l)\Big)^q\le C\sum_{I\in D_{k-s}}\nu(I)^q\\
 &\implies \hspace{1cm}\frac{-\log\sum_{I\in D_k}\mu(I)^q}{k} \ge\frac{-\log\Big( C\sum_{I\in D_{k-s}}\nu(I)^q\Big)}{k-s},
 \end{align*}
 where $C$ is a suitable constant. Applying $\liminf_{k\rightarrow\infty}$ on both sides, we get $\dim_{L^q}(\mu)\ge\dim_{L^q}(\nu).$ 
 In another case, suppose that $\beta$ is not a dyadic endpoint. Then, for some $s>0,$ we have $$\frac{j_0}{2^s}<\beta<\frac{j_0+1}{2^s}.$$
 This gives $$\beta I=\beta\Big[\frac{j}{2^{k}},\frac{j+1}{2^{k}}\Big)\subset\Big[\frac{jj_0}{2^{k+s}},\frac{jj_0+j_0}{2^{k+s}}\Big)\cup\Big[\frac{jj_0+j_0}{2^{k+s}},\frac{jj_0+j}{2^{k+s}}\Big)\cup\Big[\frac{jj_0+j}{2^{k+s}},\frac{jj_0+j+j_0+1}{2^{k+s}}\Big),$$
that is, $\beta I$ is contained in the finite number of dyadic intervals of level $k+s.$ Following the same lines, as done in the first case, we get $\dim_{L^q}(\nu)\ge \dim_{L^q}(\mu).$ Further, 
$$I=\Big[\frac{j}{2^k},\frac{j+1}{2^k}\Big)=\beta\Big[\frac{j}{\beta2^k},\frac{j+1}{\beta2^k}\Big).$$
Since $\frac{2^s}{j_0+1}<\frac{1}{\beta}<\frac{2^s}{j_0},$ we have $$\Big[\frac{j}{\beta2^k},\frac{j+1}{\beta2^k}\Big)\subset\Big[\frac{j}{(j_0+1)2^{k-s}},\frac{j+1}{(j_0+1)2^{k-s}}\Big)\cup\Big[\frac{j+1}{(j_0+1)2^{k-s}},\frac{j}{j_02^{k-s}}\Big)\cup\Big[\frac{j}{j_02^{k-s}},\frac{j+1}{j_02^{k-s}}\Big).$$
This shows that the interval $\Big[\frac{j}{\beta2^k},\frac{j+1}{\beta2^k}\Big)$ is contained in some finite number of dyadic intervals of level $k-s,$ suppose $I_1,I_2,\ldots,I_{l_*}.$ Consequently, $$I\subset \beta I_1\cup\beta I_2\cup\ldots\cup \beta I_{l_*}.$$ Following the similar approach, we get $\dim_{L^q}(\mu)\ge\dim_{L^q}(\nu).$ This completes the proof.

\end{proof}
The upcoming theorems interestingly explore the approximation aspects of a Borel probability measure with a sequence of Borel probability measures  preserving
\begin{itemize}
    \item $L^q$ dimension.
    \item Rajchman measures.
    \item $L^q$ dimension and Rajchman.
    \item $L^q$ dimension and Fourier transform with power Fourier decay.
\end{itemize}
\noindent Readers obviously think that, why we are interested in highlighting and studying some specific dense subsets of the space of Borel probability measures concerning the $L^q$ dimension and Rajchman or power Fourier decay. The Rajchman character of a measure essentially provides information about local regularity and is closely connected to the dimension theory. The Riemann-Lebesgue lemma says that a measure $\eta\in\mathcal{P}(\mathbb{R}^m)$ is Rajchman whenever it is absolutely continuous with respect to the Lebesgue measure; refer to \cite{AR3}, and the references cited therein. Thus, if $\eta$ is not Rajchman, then it is singular. As an example, the Bernoulli convolutions, except for the particular set of parameters, satisfy the Rajchman property. If for some constant $C,$ we have $|\hat{\eta}(x)|\le C|x|^{-\frac{t}{2}},$ for all $x\in\mathbb{R}^m,$ then any set which supports $\eta $ has the Hausdorff dimension atleast $t.$ In fractal geometry, the classification or identification of a class of measures or IFS satisfying specific properties is always of great interest to researchers. Recently, B\'ar\'any and K\"aenm\"aki \cite{BB1} showed that self-similar sets on $\mathbb{R}$ which neither satisfy ESC nor have exact overlaps are uncountable. Olsen \cite{O1} constructed a $G_\delta$ dense set in the space of the Borel probability measure on a compact subset of $\mathbb{R}^m$ with upper $L^q$ dimension equal to zero with respect to metric $d_L$. Recent researches clearly indicate that we want to explore the space $\mathcal{P}(\mathbb{R}^m) $ and are curious to know whether almost (or typically) all Borel probability measures are Rajchman, or all self-affine measures are Rajchman, and so on. Rapaport \cite{AR3} commented on the Rajchman property of self-similar measures on $\mathbb{R}^m.$ So, going through all the recent works, we thought, can we get some dense subsets of $\mathcal{P}(\mathbb{R}^m)$ preserving $L^q$ dimension or Rajchman property? Our upcoming theorems answer these queries interestingly. Now, the future remarks what we can comment on all these upcoming dense sets to be $G_\delta$ dense.
\begin{theorem}\label{densethm112}
    For each $\beta \in [0,m]$, the set $\mathcal{S}_\beta:=\{\mu\in\mathcal{P}(\mathbb{R}^m): \dim_{L^q}(\mu)= \beta\}$ is dense in $\mathcal{P}(\mathbb{R}^m)$.
\end{theorem}
\begin{proof}
 Let $ \mu \in \mathcal{P}(\mathbb{R}^m)$ and $\epsilon>0.$ Using the density of $\mathcal{S}(\mathbb{R}^m)$ in $\mathcal{P}(\mathbb{R}^m)$ (see \cite{Bill}), there exists $\nu=\sum_{j=1}^nc_j\delta_{x_j}$ in $\mathcal{S}(\mathbb{R}^m)$ such that $$ d_L(\mu, \nu) < \frac{\epsilon}{2}.$$ Further, we consider a non-zero compactly supported measure $ \eta \in \mathcal{S}_{\beta}$ with its support $K \subset \mathbb{R}^m.$ Let $\eta'= \nu * \eta_1,$ where $\eta_1(D):=\eta(\frac{8\sqrt{m}diam(K)D}{\epsilon})$ and $diam(K):=\sup\{\|x-y\|_2: x,y \in K\}.$ Note that Lemmas \ref{convothm} and \ref{1179} implies that $\dim_{L^q}(\eta')=\dim_{L^q}(\eta_1)=\dim_{L^q}(\eta)=\beta,$ that is, $\eta' \in \mathcal{S}_{\beta}.$ In view of the properties of convolution mentioned in Section \ref{Sec2}, we get
      \begin{align*}
       d_L(\nu,\eta')& =\sup_{f\in \text{Lip}_1} \Big\{\Big|\int_{\mathbb{R}^m} f d\nu -\int_{\mathbb{R}^m} fd\eta' \Big|\Big\} \\
      & =\sup_{f\in \text{Lip}_1} \Big\{\Big|\int_{\mathbb{R}^m} f d\Big(\sum_{j=1}^nc_j\delta_{x_j}\Big) -\int_{\mathbb{R}^m} fd(\nu*\eta_1) \Big|\Big\}.
      \end{align*}
Notice that 
\small{\begin{align*}
  \int_{\mathbb{R}^m} fd(\nu*\eta_1)&=\int_{\mathbb{R}^m}\int_{\mathbb{R}^m}f(x+y)d\nu(x)d\eta_1(y)=\int_{\mathbb{R}^m}\int_{\mathbb{R}^m}f(x+y)d(\sum_{j=1}^nc_j\delta_{x_j})(x)d\eta_1(y)\\&=\sum_{j=1}^nc_j\int_{\mathbb{R}^m}f(x_j+y)d\eta_1(y).  
\end{align*}}
Thus,
      \begin{align*}
      d_L(\nu,\eta') & = \sup_{f\in \text{Lip}_1} \Big\{\Big| \sum_{j=1}^n c_j f(x_j) -\sum_{j=1}^n c_j \int_{\mathbb{R}^m} f(y+x_j)d\eta_1(y)\Big|\Big\} \hspace{.4cm} \\
       & = \sup_{f\in \text{Lip}_1} \Big\{\Big| \sum_{j=1}^n c_j \int_{\mathbb{R}^m} f(x_j)d\eta_1(y) -\sum_{j=1}^n c_j \int_{\mathbb{R}^m} f(y+x_j)d\eta_1(y)\Big|\Big\}\\
       & \le \sup_{f\in \text{Lip}_1} \Big\{ \sum_{j=1}^n c_j  \int_{\mathbb{R}^m} |f(x_j)-f(y+x_j)|d\eta_1(y) \Big\}\\ 
       & \le  \sup_{f\in \text{Lip}_1} \Big\{  \sum_{j=1}^n c_j  \int_{\mathbb{R}^m} \text{Lip}(f)\|y\|_2d\eta_1(y)\Big\}.
       \end{align*}
       We claim that the support of $\eta_1$ is a subset of $\big[\frac{-\epsilon}{4\sqrt{m}},\frac{\epsilon}{4\sqrt{m}}\big]^m.$ Let $C\subset \mathbb{R}^m\setminus\big[\frac{-\epsilon}{4\sqrt{m}},\frac{\epsilon}{4\sqrt{m}}\big]^m.$ If $x\in C,$ then $\|x\|_2>\frac{\epsilon}{4\sqrt{m}}$ and $\frac{8\sqrt{m}diam(K)\|x\|_2}{\epsilon}>2diam(K).$ So, $\eta_1(C)=\eta(\frac{8\sqrt{m}diam(K)C}{\epsilon})=0.$ The previous estimation is reduced to
       \begin{align*}
       d_L(\nu,\eta') &  \le  \sup_{f\in \text{Lip}_1} \Big\{  \int_{\big[\frac{-\epsilon}{4\sqrt{m}},\frac{\epsilon}{4\sqrt{m}}\big]^m} \text{Lip}(f)\|y\|_2d\eta_1(y)\Big\}\\& \le\sup_{f\in \text{Lip}_1} \Big\{   \int_{\big[\frac{-\epsilon}{4\sqrt{m}},\frac{\epsilon}{4\sqrt{m}}\big]^m} \text{Lip}(f)\frac{\epsilon}{2} d\eta_1(y) \Big\}
          \le \frac{\epsilon}{2}.
      \end{align*} 
      Using the above estimate, we immediately get $$ d_L(\mu,\eta') \le  d_L(\mu,\nu) + d_L(\nu,\eta')< \frac{\epsilon}{2} + \frac{\epsilon}{2}  =\epsilon,$$ 
              showing that $\mathcal{S}_{\beta}$ is dense. 

\end{proof}



\begin{proposition}\label{convothm112}
\noindent
\begin{itemize}
    \item[(i)] Let $\mu,\nu\in\mathcal{P}(\mathbb{R}^m).$ If $\mu $ is Rajchman (with power Fourier decay), then $\mu*\nu$ is also a Rajchman (with power Fourier decay).
    \item[(ii)] Let $\mu\in\mathcal{P}(\mathbb{R}^m)$ be any Rajchman measure. Let $\beta>0$ be any real number and $\eta(D)= \mu(\beta D)$ for each Borel set $D$. Then, 
    $$\lim_{|\zeta| \to \infty}|\hat{\eta}(\zeta)|= 0.$$
\end{itemize}
\end{proposition}
\begin{proof}
It is well-known (and obvious) that the Fourier transform of a finite measure is bounded. As $\nu\in\mathcal{P}(\mathbb{R}^m),$ there exists $K\in \mathbb{R}$ such that $|\hat{\nu}(\zeta)|\le K,$ for all $\zeta\in\mathbb{R}^m.$ Consider
  \begin{align*}
      \lim_{|\zeta| \to \infty}|\widehat{\mu*\nu}(\zeta)|= \lim_{|\zeta| \to \infty}|\hat{\mu}(\zeta)\hat{\nu}(\zeta)|=\lim_{|\zeta| \to \infty}|\hat{\mu}(\zeta)||\hat{\nu}(\zeta)|\le K\lim_{|\zeta| \to \infty}|\hat{\mu}(\zeta)|=0.
  \end{align*}
  Again, note that 
  \begin{align*}
     \lim_{|\zeta| \to \infty}|\hat{\eta}(\zeta)|= \lim_{|\zeta| \to \infty}\Big|\int e^{-2\pi i\zeta.x}d\eta(x)\Big|=\lim_{|\zeta| \to \infty}\Big|\int e^{-2\pi i\zeta.\beta x}d\mu(x)\Big|=0,
  \end{align*}
  as $\mu$ is Rajchman, concluding the assertion.
\end{proof}
\begin{theorem}
    The set $\mathcal{S}_R:=\{\mu\in\mathcal{P}(\mathbb{R}^m): \lim_{|\zeta| \to \infty}|\hat{\mu}(\zeta)|= 0\}$ is dense in $\mathcal{P}(\mathbb{R}^m)$. Moreover, for each $\beta \in [0,m]$, the sets $\mathcal{S}_R^{\beta}:=\{\mu\in\mathcal{P}(\mathbb{R}^m): \dim_{L^q}(\mu)=\beta \text{ and } \lim_{|\zeta| \to \infty}|\hat{\mu}(\zeta)|= 0\}$ and $\mathcal{S}_F^{\beta}:=\{\mu\in\mathcal{P}(\mathbb{R}^m): \dim_{L^q}(\mu)=\beta \text{ and } \hat{\mu}(\zeta) \text{ has power Fourier decay} \}$ are dense in $\mathcal{P}(\mathbb{R}^m)$.
\end{theorem}
\begin{proof}
 Let $ \mu \in \mathcal{P}(\mathbb{R}^m)$ and $\epsilon>0.$ Using the density of $\mathcal{S}(\mathbb{R}^m)$ in $\mathcal{P}(\mathbb{R}^m)$, there exists $\nu=\sum_{j=1}^nc_j\delta_{x_j}$ in $\mathcal{S}(\mathbb{R}^m)$ such that $$ d_L(\mu, \nu) < \frac{\epsilon}{2}.$$Choose a non-zero compactly supported Rajchman measure $ \eta \in \mathcal{S}_R$ with its support $K \subset \mathbb{R}^m.$ Let $\eta'= \nu * \eta_1,$ where $\eta_1(D):=\eta(\frac{8\sqrt{m}diam(K)D}{\epsilon})$ and $diam(K):=\sup\{\|x-y\|_2: x,y \in K\}.$ Note that Proposition \ref{convothm112} implies that $\eta' \in \mathcal{S}_R,$ that is, $\eta'$ is Rajchman. Following the lines of Theorem \ref{densethm112}, one gets $d_L(\mu,\eta')<\epsilon,$ showing that $\mathcal{S}_R$ is dense in $\mathcal{P}(\mathbb{R}^m). $ Further, to show $\mathcal{S}_R^\beta$ is dense in $\mathcal{P}(\mathbb{R}^m), $ choose $\eta\in \mathcal{S}_R^\beta.$ This concludes the proof.
\end{proof}
\section{Conclusion \& future remarks}\label{Sec4}
Initially, we discussed density results w.r.t. the Assouad dimension and $L^q$ dimensions of sets and measures. Next, we define the modified ESC for the class of IFS consisting of a finite number of affine or conformal transformations on $\mathbb{R}^m$. Thus, our definition is the weaker version of the ESC, defined in \cite{MHOCHMAN1}. In Theorem \ref{th3.4}, we established ``modified ESC$\implies$ESC, in the case of IFS consisting of similarity transformations in $\mathbb{R}^m.$'' But the converse is still open to us. A step towards the converse is shown in Proposition \ref{prop1} in the case when all the similarity ratios are equal in $\mathbb{R}.$ It is noted that this case covers a strong class of IFS such as Cantor-type sets, Bernoulli convolutions.  Olsen, in his work on $L^q$ dimensions of measure \cite{O1}, proved the $G_{\delta}$-dense subset of the space of Borel probability measures with respect to the Monge-Kantorovich metric. Analogous to the decomposition of a function preserving dimension, the decomposition of measures under convolution preserving $L^q$ dimension is not possible in general. For any real number $\alpha$ with $0\le \alpha \le m$ and any measure $\mu \in \mathcal{P}( \mathbb{R}^m)$, in general there do not exist $\nu,\eta\in \mathcal{P}( \mathbb{R}^m)$ such that $\mu=\nu * \eta$ and $\dim_{L^q}(\nu)=\dim_{L^q}(\eta)=\alpha.$ This is because of Lemma 2.2 of Feng et al. \cite{Feng}, which claims $\dim_{L^q}(\mu)\ge \alpha.$ But we can try to prove the following restricted version of decomposition and dimension-based results in the future.
\\

\textbf{Issue 1.}
Let $\mu \in \mathcal{P}( \mathbb{R}^m)$ and $\alpha\in \mathbb{R}$ with  $0\le \alpha \le \dim_{L^q}(\mu).$ Then, there exist Borel probability measures $\nu,\eta$ such that $\mu=\nu * \eta$ and $\dim_{L^q}(\nu)=\dim_{L^q}(\eta)=\alpha.$\\

\textbf{Issue 2.}
Let $0\le \alpha \le m$, and $\mu$ be a finite Borel measure such that $\dim_{L^q}(\mu) \le \alpha.$ Then, there exist finite Borel measures $\nu,\eta$ such that $\mu=\nu+\eta$ and $\dim_{L^q}(\nu)=\dim_{L^q}(\eta)=\alpha.$

\section*{Statements and Declarations}
 The authors have no competing interests to declare that are relevant to the content of this article. There are no data associated with this paper.
   
\section*{Acknowledgements}
 The first author is supported by the SEED grant project of IIIT Allahabad. The second author is financially supported by the Ministry of Education at the IIIT Allahabad. Some parts of the paper have been presented in lectures delivered by the first author at the workshop ``Course on Exponential Separation and $L^q$ Dimension" organized by IIIT Allahabad on April 23, 2025.

\bibliographystyle{amsalpha}

\begin{thebibliography}{A}
\bibitem{Baker1} S. Baker, Iterated function systems with super-exponentially close cylinders, Adv. Math. 379 (2021) 107548.

\bibitem{FB} F. Bayart, How do the typical $L^q$-dimensions of measures behave?, Indiana Univ. Math. J. 63(3) (2014) 687-726.

\bibitem{BB1} B. Bárány, A. Käenmäki, 
Super-exponential condensation without exact overlaps, Adv. Math. 379 (2021) 107549.





\bibitem{BB2} B. Bárány, M. Verma, Hausdorff dimension of self-similar measures and sets with common fixed point structure, Preprint, arXiv:2507.05835 (2025), 27 pages.
\bibitem{BB} B. Bárány, A. Käenmäki, H. Koivusalo, Dimension of self-affine sets for fixed translation vectors, J. Lond. Math. Soc. 98(1) (2018) 223-252.
\bibitem{BB3} B. Bárány, I. Kolossváry, S. Troscheit, On exponential separation of analytic self-conformal sets on the real line, Preprint, arXiv:2509.07888 (2025).




\bibitem{Bed} T. Bedford, Crinkly curves, Markov partitions and box dimensions in self-similar sets, PhD thesis, University of Warwick, 1984.
\bibitem{Bill} P. Billingsley, Convergence of probability measures, Wiley \& Sons, New York-London, 1968.


\bibitem{Chen} C. Chen, M. Wu, W. Wu, Accessible values for the Assouad and lower dimensions of subsets, Real Anal. Exchange 45(1) (2020) 85-100. 

\bibitem{Chen1} C. Chen, Self-similar sets with super-exponential close cylinders, Ann. Fenn. Math. 46(2) (2021) 727-738.

\bibitem{DGSH} A. Deliu, J. S. Geronimo, R. Shonkwiler, D. Hardin, Dimensions associated with recurrent self-similar sets, Proc. Cambridge Philos. Soc. 110(2) (1991), 327--336. 

\bibitem{Fal} K. J. Falconer, Fractal Geometry: Mathematical Foundations and Applications, John Wiley Sons Inc., New York, 1999.

\bibitem{Fal1} K. J. Falconer, The Hausdorff dimension of self-affine fractals, Math. Proc. Camb. Philos. Soc. 103(2) (1988) 339-350. 

\bibitem{Feng} D. J. Feng, N. T. Nguyen, T. Wang, Convolutions of equicontractive self-similar measures on the line, Illinois J. Math. 46(4) (2002) 1339-1351.

\bibitem{Fraser} J. M. Fraser, Assouad dimension and fractal geometry, Vol. 222, Cambridge University Press, 2020.

\bibitem{Fraser1} J. M. Fraser, Fractal Geometry of Bedford-McMullen Carpets, Lecture Notes in Mathematics, Vol 2290. Springer, 2021.  

\bibitem{Garsia} A. M. Garsia, Arithmetic properties of Bernoulli convolutions, Trans. Amer. Math. Soc. 102(3) (1962) 409-432.

\bibitem{MHOCHMAN} M. Hochman, On self-similar sets with overlaps and inverse theorems for entropy, Ann. Math. 180(2) (2014) 773-822.

\bibitem{MHOCHMAN1} M. Hochman, On self-similar sets with overlaps and inverse theorems for entropy in $\mathbb{R}^d,$ American Mathematical Society, 2022.

\bibitem{Hochman2} M. Hochman, A. Rapaport, Hausdorff dimension of planar self-affine sets and measures with overlaps, J. Eur. Math. Soc. 24(7) (2022) 2361-2441.

 
\bibitem{H} J. E. Hutchinson, Fractals and self similarity, Indiana Uni. Math. J. 30(5) (1981) 713-747.

\bibitem{JS} S. Jaffard, S. Seuret, H. Wendt, R. Leonarduzzi, P. Abry, Multifractal formalisms for multivariate analysis, Proc. R. Soc. A 475(2229) (2019) 20190150.

\bibitem{Kaufman} R. Kaufman, On Hausdorff dimension of projections, Mathematika 15(2) (1968) 153-155.
\bibitem{IDM} I. D. Morris, P. Shmerkin, On equality of Hausdorff and affinity dimensions, via self-affine measures on positive subsystems, Trans. Amer. Math. Soc. 371 (2019) 1547-1582.
 
\bibitem {Mat3} P. Mattila, Fourier analysis and Hausdorff dimension, Cambridge University Press, Cambridge, 2015.
 



\bibitem{M} P. A. P. Moran, Additive functions of intervals and Hausdorff measure, Proc. Cambridge Philos. Soc. 42 (1946) 15-23.

\bibitem{O1} L. Olsen, Typical $L^q$-dimensions of measures, Monatsh. Math. 146(2) (2005) 143-157.

\bibitem{O2} L. Olsen, On the average $L^q$-dimensions of typical measures, Indiana Univ. Math. J. 68(3) (2019) 697-720.

\bibitem{O3} L. Olsen, On the average $L^q$-dimensions of typical measures belonging to the Gromov-Hausdorff-Prohoroff space. The limiting cases: $q=1$ and $q=\infty$, Ann. Acad. Sci. Fenn. Math. 45(2) (2020) 647-672.

\bibitem{Parth} K. R. Parthasarathy, Probability measures on metric spaces, Academic Press, New York-London, 1967.

\bibitem{Peres} Y. Peres, B. Solomyak, Existence of $L^q$ dimensions and entropy dimension for self-conformal measures, Indiana Univ. Math. J. 49(4) (2000) 1603-1621.


\bibitem{AR1} A. Rapaport, Proof of the exact overlaps conjecture for systems with algebraic contractions, 
Ann. Sci. Ecole. Norm. S. 55(5) (2022) 1357-1377.

\bibitem{AR3} A. Rapaport, On the Rajchman property for self-similar measures on $\mathbb{R}^d,$ Adv. Math. 403 (2022) 108375.

\bibitem{AR2} A. Rapaport, P. P. Varj\'u, Self-similar measures associated to a homogeneous system of three maps, Duke Math. J. 173(3) (2024) 513-602.

\bibitem{AR4} A. Rapaport, Dimension of self-conformal measures associated to an exponentially separated analytic IFS on $\mathbb {R} $, Preprint, arXiv:2412.16753 (2024).

\bibitem{Re} A. Rényi, On measures of entropy and information, In Proceedings of the fourth Berkeley symposium on mathematical statistics and probability, Vol. 4 (1961) 547-562.

\bibitem{Schief} A. Schief, Separation properties for self-similar sets, Proc. Amer. Math. Soc. 122(1) (1994) 111--115. 
\bibitem{Sh1} P. Shmerkin, On Furstenberg's intersection conjecture, self-similar measures, and the $L^q$ norms of convolutions, Ann. of Math. 189(2) (2019) 319-391.

\bibitem{BS1} B. Solomyak, Measure and dimension for some fractal families, Math. Proc. Camb. Philos. Soc. 124(3) (1998) 531-546.

\bibitem{BS} B. Solomyak, Fourier decay for self-similar measures, Proc. Amer. Math. Soc. 149(8) (2021) 3277-3291.
\bibitem{BS12} B. Solomyak, Fourier decay for homogeneous self-affine measures, J. Fractal Geom. 9(1) (2022) 193-206.



\bibitem{PPV1} P. P. Varj\'u, On the dimension of Bernoulli convolutions for all transcendental parameters, Ann. of Math. 189(3) (2019) 1001-1011.





\bibitem{VP1} S. Verma, A. Priyadarshi, Further analysis of Hausdorff dimension and separation conditions, Contemporary Mathematics (AMS), Vol. 825 (2025) 225-239.

\bibitem{VM1} S. Verma, P. R. Massopust, Dimension preserving approximation, Aequat. Math. 96(6) (2022) 1233-1247.


\end{thebibliography}

\end{document}